\documentclass{amsart}

\usepackage{mathtools}
\usepackage[utf8]{inputenc}
\usepackage[english]{babel}
\usepackage[]{amssymb}
\usepackage{amsmath}
\usepackage{amsfonts}
\usepackage{graphicx}
\usepackage{tikz-cd}

\usepackage[style=alphabetic, backend=biber]{biblatex}
\usepackage{csquotes}
\usepackage{float}
\usepackage[
    colorlinks=false,
    linkbordercolor={0.608 0.310 0.588},  
    citebordercolor={0.839 0.008 0.439},  
    urlbordercolor={0.000 0.220 0.659},   
    pdfborder={0 0 1}                     
]{hyperref}

\DeclareGraphicsExtensions{.pdf,.png,.jpg}
\DeclareMathOperator{\lcm}{lcm}
\DeclareMathOperator{\Ext}{Ext}
\DeclareMathOperator{\Hom}{Hom}
\DeclareMathOperator{\RR}{R}
\DeclareMathOperator{\Supp}{Supp}

\begin{document}

\newcommand{\C}{\mathbb C}
\newcommand{\Z}{\mathbb Z}
\renewcommand{\P}{\mathbb P}
\newcommand{\Q}{\mathbb Q}
\newcommand{\R}{\mathbb R}
\newcommand{\N}{\mathbb N}
\newcommand{\Fp}{\mathbb{F}_p}
\newcommand{\Fq}{\mathbb{F}_q}

\newtheorem{thm}{Theorem}[section]
\newtheorem*{thm*}{Theorem}
\newtheorem{prop}[thm]{Proposition}
\newtheorem{lem}[thm]{Lemma}
\newtheorem{cor}[thm]{Corollary}
\newtheorem*{cor*}{Corollary}
\theoremstyle{definition}
\newtheorem{definition}[thm]{Definition}
\newtheorem{condition}[thm]{Condition}
\newtheorem{example}[thm]{Example}
\theoremstyle{remark}
\newtheorem{remark}[thm]{Remark}
\numberwithin{equation}{section}

\title[Intersection form for hypersurfaces of weighted projective space ]{Cohomology of hypersurfaces of weighted projective space and the intersection form on $H^2$}
\author{Anna-Maria Raukh}

\begin{abstract}
Given a hypersurface $X \xhookrightarrow{i} \widetilde{P}^n$ in a weighted projective space, we compute the intersection form on the second cohomology $H^2(X, \Z)^{\otimes n-1} \to \Z$ for the purpose of identifying Fano manifolds obtained from smoothing singular Fanos. In the process, we describe the integer cohomology groups $H^k(X, \Z)$ for $k<n$ and give an explicit formula for the pullback map $i^*$.
\end{abstract}

\maketitle

\section{Introduction}
There is substantial evidence that many Fano manifolds can be obtained by smoothing singular toric Fano varieties, or by smoothing toroidal crossing spaces, e.g., 
\cite{MR4304077},
\cite{corti2024smoothinggorensteintoricfano}, 
\cite{Corti_2022}, 
\cite{corti2024makelogstructures}, 
\cite{corti2025singularlogstructureslog}, 
\cite{grafnitz2025smoothingszeromutablelaurent}, 
\cite{doran2023modularitylandauginzburgmodels}.
In order to identify the Fano manifold obtained from a smoothing, one needs suitable invariants and the intersection form on the second cohomology is a good candidate. For example, a Fano threefold hypersurface can be uniquely identified from its triple intersection form together with the index.
This work is part of a project to compute the 3-form on the second coholomology of all Fano threefolds as well as known higher dimensional Fano manifolds. We focus here on hypersurfaces in weighted projective spaces. Fanos of such form are of current interest \cite{kollar2001}, \cite{campo2023blowupssmoothfanohypersurfaces}, \cite{sano2025deltainvariantsweightedhypersurfaces}. Pieter Belmans kindly informed us that this intersection 3-form already featured in Mori—Mukai "Classification of Fano 3-folds with $B_2\geq 2$" \cite{Mori1981}
and was used by Sergey Galkin \cite{galkin2018smalltoricdegenerationsfano}.

It is straightforward to compute this 3-form on Fano manifolds which are hypersurfaces in ordinary projective space by using the Lefschetz hyperplane theorem \cite{Milnor}. The hyperplane theorem states that, for a general hyperplane section of an $n$-dimensional projective variety, pullback under the inclusion map is an isomorphism for cohomology groups up to degree $n-2$, and a surjection in degree $n-1$. Using the Veronese embedding, one can consider a regular hypersurface as a hyperplane section. Therefore, we may apply the Lefschetz theorem and get that the 3-form on the hypersurface is the same as the intersection form of the ambient space just multiplied by the degree of the hypersurface. 

However, the classic Lefschetz theorem does not hold for weighted projective spaces, because it requires the complement of a hypersurface to be smooth, but weighted projective spaces are singular.
Weighted projective spaces are toric varieties and Batyrev and Cox \cite{BatyrevCox} generalized the Lefschetz theorem to hypersurfaces in toric varieties, but only for the field coefficients while we need the result for integer coefficients. A similar statement can be made about the work by Mavlyutov \cite{mavlyutov} who considered complete intersections. 

\newpage

We provide an analogue of Lefschetz theorem for hypersurfaces in weighted projective spaces and cohomology with integer coefficients: 
\begin{thm*}[Theorem \ref{myt} below] 
    Let $\widetilde{\P}^n = \P(q_0,\ldots,q_n)$ be a weighted projective space with pairwise coprime weights. Let $X$ be a  
 general hypersurface of $\widetilde{\P}^n$ of degree $d$ which is divided by all $q_i$. And let $ i \colon X \rightarrow \widetilde{\P}^n$ be the inclusion map. Then for $k < n-1$ 
    \begin{enumerate}
        \item \begin{equation*}
    H^k(X,\Z) = 
    \begin{cases}
            \Z, & \text{if $k$ is even}\\
            0, & \text{otherwise}
    \end{cases}
\end{equation*}
        \item for even $k=2r$ the pullback map $$i^* \colon H^{2r}(\widetilde{\P}^n, \Z) \rightarrow H^{2r}(X, \Z)$$ is multiplication by $\frac{l_r \cdot l_{n-r}}{l_n}$, where $l_i$ is the integer derived from weights $(q_0,\ldots,q_n)$ (for precise definition see \ref{def:li}).

    \end{enumerate}
\end{thm*}
As one can see, the pullback is no longer an isomorphism. This occurs because relative Poincaré duality — which underpins the Lefschetz theorem — requires field coefficients and smooth manifolds, where the cup product between cohomology groups in complementary degrees gives a perfect pairing.  In our case, it does not. Instead, we use the extraordinary cup product for cohomology with support and Alexander duality with non-field coefficients.

We use this result to compute the intersection form on second cohomology of Fano hypersurfaces in weighted projective spaces of any dimension:

\begin{cor*}[Corollary \ref{c27} below]
    For a general smooth hypersurface $X$ of degree $d$ in $\P(q_0, \ldots,  q_n)$, the $n$-form on the second cohomology induced by the cup product is given by:
    \begin{equation*}
        (\alpha_1, \ldots, \alpha_n) \to \frac{l_{n+1}^{n-1} \cdot d}{l_n^n} \prod\limits_{i=1}^n \alpha_i
    \end{equation*}
\end{cor*}

Table \ref{table:1} presents the intersection forms and indices for all smooth weighted Fano hypersurfaces in low dimensions. In dimensions 3, 4, and 5, the intersection form and index uniquely distinguish each hypersurface. However, in dimension 6, we exhibit an example of two hypersurfaces with different topology that share the same intersection form and index, indicating that these invariants are insufficient to uniquely identify Fano hypersurfaces in higher dimensions.

I am grateful to Pieter Belmans, Thomas Guidoni, Dmitrii Krekov, Dmitry Mineev, Erik Paemurru, David Ploog, Helge Ruddat and Konstantin Shramov for useful discussions. I acknowledge the financial support of the University of Stavanger.

\section{Cohomology of weighted projective spaces}

We will only deal with spaces over the field of complex numbers $\C$. We will denote the multiplicative group of this field by $\C^*$.

\begin{definition}[Weighted projective space]

Let $q=(q_0,\ldots,q_n)$ with $q_i\in\N$ and define the corresponding action of $\C^*$ on $\C^{n+1}\setminus\{0\}$ by

$$\lambda\cdot(x_0,\ldots,x_n)=(\lambda^{q_0}x_0,\ldots,\lambda^{q_n}x_n).$$

The quotient of this action is a \textit{weighted projective space}. It is a projective variety. We will denote it as $\P(q_0,\ldots,q_n)$ or $\widetilde{\P}^n$.

\end{definition}

\begin{example}
    $\P(1,\ldots, 1)$ is the ordinary complex projective space.
\end{example}

As in ordinary projective space, we can use homogeneous coordinates $(x_0:\ldots:x_n)$ and define a \textit{hypersurface of degree} $d$ as the zero locus of a weighted homogeneous polynomial of a weighted degree $d$.

We are interested in integer cohomology of hypersurfaces in weighted projective spaces. We first need to understand the cohomology ring of the weighted projective space itself.

To describe the cohomology ring of a weighted projective space, we need to introduce one combinatorial definition. 

\begin{definition}\label{def:li}
    For each subset $I = \{i_0,\ldots, i_k\}$ of $\{0, \ldots, n\}$ denote 
    \begin{equation*}
        l_I = \frac{i_0 \cdot \ldots \cdot i_k} {\gcd (i_0, \ldots , i_k)}
    \end{equation*}

    Define $l_r$ as
    \begin{equation*}
        l_r = \lcm \{l_I \mid I \subset \{0,\ldots,n \} \text { and } |I| = r+1 \}
    \end{equation*}
\end{definition}

\begin{example}
    $l_0 = 1$, $l_1 = \lcm \{q_0, \ldots, q_n\}$, $l_n = q_0 \cdot \ldots \cdot q_n / \gcd (q_0, \ldots, q_n)$
\end{example}

With this definition in place, we now recall two theorems from \cite{amrani}, which completely describe the cohomology ring of weighted projective spaces.

\begin{thm}\label{t1} For a weighted projective space $\widetilde{\P}^n$ we have
\begin{equation*}
    H^i(\widetilde{\P}^n,\Z) \simeq
    \begin{cases}
            \Z, & \text{if $i=2r$, $0 \leq r \leq n$}\\
            0, & \text{otherwise}
    \end{cases}
\end{equation*}

\end{thm}

\begin{proof}
    \cite{amrani}
\end{proof}

Theorem \ref{t1} tells us that the cohomology groups of a weighted projective space are equal to the cohomology groups of the usual projective space. The next theorem will describe how exactly they map to each other.

\begin{thm}\label{t2}
    Let $\phi \colon \P^n \to \widetilde{\P}^n$ be the map that sends each point $(x_0: \ldots :x_n)$ to $\phi (x_0: \ldots :x_n) = (x_0^{q_1}: \ldots :x_n^{q_n})$. Let $\xi = c_1(L^*)$ be the usual generator of $H^2(\P^n, \Z)$ ($L$ is the canonical bundle over $\P^n$). 
    Then for each $k$, $0 \leq k \leq n$, there exists a unique $\xi_k \in H^{2k}(\widetilde{\P}^n, \Z)$ such that the pullback $\phi^*(\xi_k) = l_k \xi^k$, and $\{ \xi_1, \ldots, \xi_n \}$ form a graded $\Z$-basis of the free abelian group $H^*(\widetilde{\P}^n, \Z)$. In other words, there are commutative diagrams:

\begin{center}

\begin{tikzcd}
H^{2k}(\widetilde{\P}^n, \Z) \arrow[d, equal] \arrow{rr}{\phi^*} &  & H^{2k}(\P^n, \Z) \arrow[d, equal] \\
\Z \arrow{rr}{\cdot l_r}                             &  & \Z                               
\end{tikzcd}

\end{center}
    
\end{thm}

\begin{proof}
    \cite{amrani}[I.4]
\end{proof}

\newpage

It can be seen from these two theorems that the multiplication in the ring $H^*(\widetilde{\P}, \Z)$ is of the form:
\begin{equation}\label{1}
    \xi_k \cdot \xi_j = \frac{l_r \cdot l_j}{l_{k+j}} \xi_{k+j} 
\end{equation}

We will also recall another result of Al Amrani about the degree of the map $\phi$.

\begin{thm}\label{t3}
    The degree of a map $\phi$ is $q_0 \cdot \ldots \cdot q_n / \gcd (q_0, \ldots ,q_n) = l_n$.
\end{thm}

\section{Fano hypersurfaces in weighted projective spaces}

 Weighted projective spaces are toric varieties. 

\begin{lem}[\cite{rossi2013weightedprojectivespacestoric}]
    Let $\{e_1, \ldots , e_n\}$ be the basis of the lattice $N$ and $(q_0, \ldots, q_n)$ the set of positive integers. Consider the following n + 1 rational vectors

\begin{equation*}
    v_0=-\frac{1}{q_0}\sum e_i, \quad v_i=\frac{e_i}{q_i}.
\end{equation*}
Let $N_Q$ be the lattice generated by $v_0, \ldots , v_n$ and consider the fan $\Sigma_Q$ generated by $v_0, \ldots , v_n$. Then the toric variety corresponding to $\Sigma_Q$ is the weighted projective space $\P(q_0, \ldots, q_n)$.
\end{lem}

We are interested in smooth hypersurfaces in a weighted projective space. Therefore, we need the singular locus of $\P(q_0, \ldots, q_n)$ to be 0-dimensional, so that the generic hypersurface would not meet the singularities.

\begin{prop}\label{p1}
    $\P(q_0, \ldots, q_n)$ have at worst isolated singularities if and only if weights $q_0, \ldots, q_n$ are pairwise coprime.
\end{prop}
\begin{proof}
    If there are $q_k, q_l$ with $gcd(q_k, q_l)>1$ then the $n-1$-dimensional convex cone spanned by $\langle v_0,\ldots, \check v_k, \ldots, \check v_l, \ldots, v_n \rangle$ is non-smooth. Therefore, the singularity has positive dimension. If there are no such pairs, then the only non-smooth cones are of dimension $n$. 
\end{proof}

But this condition is not enough. For some degrees, any hypersurface of this degree will pass through a singular point of $\widetilde{\P}^n$.

\begin{example}
    Let $X$ be a hypersurface of degree 3 in $\widetilde{\P}(1,1,1,2)$. In homogeneous coordinates it is given by equation 
    
    \begin{equation*}
        F(x_0,x_1,x_2) + x_3 \cdot G(x_0,x_1,x_2)=0
    \end{equation*}
    where $F, G$ are homogeneous polynomials of degrees 3 and 1, respectively. The unique singular point of $\widetilde{\P}(1,1,1,2)$ is $(0:0:0:1)$. Both $F(0,0,0)$ and $G(0,0,0)$ equal zero and therefore 
    \begin{equation*}
        F(0,0,0) + x_3 \cdot G(0,0,0)=0
    \end{equation*}
    so any hypersurface of degree 3 passes through the singular point.
\end{example}

\begin{prop}\label{p2}
    For the surface not to pass through the singularities of $\widetilde{\P}^n$, the degree should allow monomials $x_i^k$ for each non-trivial $q_i$. In other words, the degree should be divided by all weights $q_i$.
\end{prop}

Some of those smooth hypersurfaces happen to be Fano.

\begin{definition}
     A \textit{Fano manifold} is a smooth projective variety whose anticanonical bundle is ample.
\end{definition}

\begin{lem}[\cite{kollar2001}]\label{p3}
    Let $X$ be a hypersurface of degree $d$ in weighted projective space $\widetilde{\P}^n(q_0, \ldots, q_n)$ which has at worst isolated singularities. It is Fano if and only if $d < \sum q_i$.
\end{lem}

Let us now look at some examples. For 3-dimensional hypersurfaces in 4-dimensional weighted projective space combining Lemmas \ref{p1}, \ref{p2} and \ref{p3} we are left with 5 hypersurfaces: degree 2 and 4 in $\P(1,1,1,1,2)$, degree 3 and 6 in $\P(1,1,1,1,3)$, and degree 6 in $\P(1,1,1,2,3)$.
Note that the hypersurface of degree 2 in $\P(1,1,1,1,2)$ and the hypersurface of degree 3 in $\P(1,1,1,1,3)$ are just $\P^3$. So there are exactly three Fano threefolds of such form.

 
%

\section{Some homological algebra}

To overcome difficulties related to the singularities of the weighted projective space, we will use sheaf cohomology. We quote lemmas and adapt the notations from Dimca's "Sheaves in topology" \cite[chapter 1.4]{dimca}.

Let $\mathcal{A}$ be an abelian category. 
Then let $C(\mathcal{A})$ be the category of complexes, $K(\mathcal{A})$ be its homotopy category, $D(\mathcal{A})$ be the derived category. 

Let $X^{\bullet},Y^{\bullet} \in D(\mathcal{A})$. Denote by $X^{\bullet}[n]$ the shift by $n$, in other words $X^k[n] = X^{k+n}$ 

\begin{lem}\label{1.4.3}
     Then $\Ext^{n}(X^{\bullet},Y^{\bullet}) \simeq Hom_{D(\mathcal{A})}(X^{\bullet},Y^{\bullet}[n])$ for all integers $n$.
\end{lem}

\begin{lem}\label{1.4.1}
    $H^n(Hom^\bullet(X^\bullet,Y^\bullet)) \simeq Hom_{K(\mathcal{A})}(X^\bullet,Y^\bullet[n])$ for all integers n.
\end{lem}

Denote the category of modules on the ring $A$ as $mod(A)$. We will say that a complex $X^{\bullet} \in C(mod(A))$ is free (injective) if every object in it is free (injective). 

\begin{condition}\label{cond}
    Assume that either:
\begin{enumerate}
    \item $X^\bullet\in C^-(mod(A))$ and $Y^\bullet\in C^+(mod(A))$, or

    \item $X^\bullet\in C(mod(A))$ and $Y^\bullet\in C^b(mod(A))$, or

    \item one of (1) or (2) above holds when we interchange $X^\bullet$ and $Y^\bullet.$
\end{enumerate}
\end{condition}

\begin{thm}[Universal coefficients]\label{uc}

    Let $A$ be a principal ideal domain. 
    
    Let $X^{\bullet}, Y^{\bullet} \in C(mod(A))$ be complexes satisfying Condition \ref{cond}. If moreover $X^{\bullet}$ is free or $Y^{\bullet}$ is injective, then for any integer $m \in \Z$ we have the following short exact sequence:
    
\begin{multline*}
    0 \to \bigoplus_{q - p = m - 1} \Ext(H^p(X^\bullet), H^q(Y^\bullet)) \to  H^m(\Hom^\bullet(X^\bullet, Y^\bullet)) \to \\
    \bigoplus_{q - p = m} \Hom(H^p(X^\bullet), H^q(Y^\bullet)) \to 0.
\end{multline*}

In particular, if $ A \to Y^\bullet$ is a bounded injective resolution, then we have for any complex $X^\bullet$ and any $m \in \Z$ the following exact sequence:

\begin{equation*}
    0 \to \Ext^1(H^{m+1}(X^\bullet), A) \to H^{-m}(\Hom^\bullet(X^\bullet, Y^\bullet)) \to \Hom(H^m(X^\bullet), A) \to 0.
\end{equation*}  

\end{thm}

\section{Cohomology with supports in a closed subvariety}

Let $Z$ be a closed subvariety of $X$ and denote $i\colon Z \to X$ the corresponding closed immersion. Then one can define cohomology with supports in $Z$ (sometimes called \textit{local} cohomology). 

\begin{definition}
    Define $\Gamma_Z (X, \mathcal{F})$ to be the set of sections of $\mathcal{F}$ with support in $Z$.
    $$ \Gamma_Z (X, \mathcal{F}) := \{ s \in \Gamma (X, \mathcal{F}) | \Supp(s) \subseteq Z  \} $$
    Then $\Gamma_Z (X, -)$ is a left exact functor. \textit{Cohomology with supports in} $Z$ and coefficients in $\mathcal{F}$, denoted by $H_Z^m(X, \mathcal{F})$, is the  $m$-th derived functor of $\Gamma_Z (X, -)$.
\end{definition}

We will soon see how cohomology with supports in $Z$ relates to the usual cohomology $H^p(Z)$. But first we need to discuss an extraordinary cup product.

\subsection{Extraordinary cup product}

For a ring $A$ denote by $A_X$ and $A_Z$ the constant sheaves on $X$ and $Z$ respectively.

Since $\Gamma (X, \mathcal{F}) \simeq \Hom (A_X, \mathcal{F})$ we have $H^m(X, \mathcal{F}) = \Ext^m(A_X, \mathcal{F})$. A similar expression holds for cohomology with supports:

\begin{lem}\label{2.4.2}
    $\Gamma_Z(X, \mathcal{F}) \simeq Hom (i_! A_Z, \mathcal{F})$ and therefore $H^m_Z(X,\mathcal{F}) \simeq \Ext^m(i_! A_Z, \mathcal{F})$
\end{lem}

\begin{proof}
    \cite{dimca}[Lemma 2.4.2]
\end{proof}

Composition in the derived category gives rise to a pairing called the composition product:

\begin{equation*}
    \Ext^p(N, P) \times \Ext^q(M, N) \to \Ext^{p+q}(M, P)
\end{equation*}

\begin{equation*}
    \alpha \cup \beta := \alpha [q] \circ \beta.
\end{equation*}

Due to the Lemma \ref{2.4.2} we can define the cup product in the cohomology with support as a composition product:

\begin{equation*}
    \Ext^q(A_X, A_X) \times \Ext^p(i_!A_Z, A_X) \to \Ext^{p+q}(i_!A_Z, A_X)
\end{equation*}

\begin{equation*} 
\beta \cup \gamma \in H_Z^{p+q}(X, A_X), \text{  where  } \beta \in  H^q(X, A_X), \gamma \in H_Z^q(X, A_X)  
\end{equation*}

\

We can also express cohomology of $Z$ in terms of $\Ext$ of sheaves on $X$ using the following lemma:

\begin{lem}\label{9.12}
    $H^m(Z, A_Z) = Ext^m (i_! A_Z, i_! A_Z)$
\end{lem}
\begin{proof}
    \cite{iversen}[II.9.12]
\end{proof}

Similarly, we define the extraordinary cup product:

\begin{equation*}
\beta \cup \gamma \in H_Z^{p+q}(X, \mathcal{F}), \text{  where  } \beta \in H_Z^p(X, \mathcal{F}), \gamma \in H(Z, A_Z).
\end{equation*}

\

    The two cup products are related by the following formulas:

\begin{lem}
    Let $r\colon H_Z^{\bullet}(X, A_X) \to H^{\bullet}(X, A_X)$ and $i^*\colon H(X, A_X)^{\bullet} \to H(Z, A_Z)^{\bullet}$ be the map induced by closed immersion of $Z$. For $\alpha \in H^p(X, A_X)$, $\beta \in H_Z^q(X, A_X)$, $\gamma \in H^r(Z, A_Z)$ following formulas hold:

\begin{enumerate}
    \item  \begin{equation*}
        (\alpha \cup \beta) \cup \gamma = \alpha \cup (\beta \cup \gamma)
    \end{equation*} 
    \item \begin{equation*}
         r(\alpha)\cup\beta=\alpha\cup i^*r(\beta) 
    \end{equation*}
    \item \begin{equation*}
         \alpha \cup \beta = (-1)^{pq} \beta \cup i^* \alpha 
    \end{equation*}
    \end{enumerate}

\end{lem}

\begin{proof}
    \cite{iversen}[II.9]
\end{proof}

We will now prove the analogue of the projection formula.

\begin{lem}\label{pf}
    Let $\alpha \in H_Z^p(X, A_X)$ and $\beta \in H^q (X, A_X)$. Then the following formula holds:
$$ r(\alpha \cup i^* \beta) = r(\alpha) \cup \beta$$
    
\end{lem}

\begin{proof}

Let $I^{\bullet}$ be an injective resolution of the constant sheaf $A_X$ on $X$. Then denote $J^{\bullet} = i_! i^* I^{\bullet}$ and $\pi \colon  I^{\bullet} \to J^{\bullet}$ the anjunction morphism. Note that $A_Z = i^*A_X$. Now recall that
\begin{equation*}   
    H^p_Z(X,A_X) \simeq \Ext^p(i_! A_Z, A_X) \simeq \Ext^p(i_! i^* A_X, A_X) \simeq \Ext^p(J^{\bullet}, I^{\bullet})
\end{equation*}

and that

\begin{equation*}
     H^q(X,A_X) \simeq \Ext^q(A_X, A_X) \simeq \Ext^q(I^{\bullet}, I^{\bullet})
\end{equation*}
  
With these notations the map $r\colon H_Z^{\bullet}(X, A_X) \to H^{\bullet}(X, A_X)$ is simply:

\begin{equation*}
    r(\alpha) = \alpha \circ \pi .
\end{equation*}

Also, recall that the cup product comes from composition. Then using (3) from the previous lemma twice and the anticomutativity property of the cup product, we have

\begin{equation*}
         r(\alpha \cup i^* \beta) = (-1)^{pq} r(\beta \cup \alpha) = \beta [p] \circ \alpha \circ \pi = (-1)^{pq} \beta \cup r(\alpha) = r (\alpha) \cup \beta .
\end{equation*}

\end{proof}

This lemma leads to the following commutative diagram:

\begin{center}

\begin{tikzcd}
H^{2r}(X, A) \arrow{r}{i^*} \arrow[d]              & H^{2r}(Z, A) \arrow{d} \\
H^{2n - 2r}(X, A)^{\vee} \arrow{r}{r^{\vee}} & H^{2n - 2r}_Z(X, A)^{\vee}                        
\end{tikzcd}

\end{center}

where vertical maps are induced by the ordinary and extraordinary cup products. 

\subsection{Alexander duality}

We will recall the Verdier's duality first.

\begin{thm}[Verdier's duality]
    For a continuous mapping $f\colon X \to Y$ there is an additive functor of triangulated categories $f^!\colon D^+ (Y) \to D^+ (X)$, such that 
    \begin{equation*}
        RHom^{\bullet}(Rf_!\mathcal{F}^{\bullet}, \mathcal{G}^{\bullet}) \simeq RHom^{\bullet}(\mathcal{F}^{\bullet}, f^! \mathcal{G}^{\bullet})
    \end{equation*}
    in $D^+(mod(A))$. In particular, we have

    \begin{equation*}
        Hom_{D^+(Y)}(Rf_!\mathcal{F}^{\bullet}, \mathcal{G}^{\bullet}) \simeq Hom_{D^+(X)}(\mathcal{F}^{\bullet}, f^! \mathcal{G}^{\bullet}).
    \end{equation*}
    In other words, $f^!$ is right adjoint to the $Rf_!$.
\end{thm}

\begin{remark}
    Verdier's duality is a more general form of Poincare duality and is compatible with the cup product defined with $\Ext$ as above. 
\end{remark}

Denote $a_X: X \to pt$ and $\omega_X = a_X^!(A)$. For an orientable manifold $\omega_X \simeq A_X[n]$.

\begin{thm}[Alexander duality with non-field coefficients]\label{AD}
    Let $A$ be a principal ideal domain and $X$ be an orientable manifold of dimension $n$. And let $i\colon Z \to X$ be an inclusion of a closed subset. Then for any $m \in \Z$ one has the following exact sequence:
\begin{equation*}
     0 \to \Ext (H^{n-m+1}(Z, A_Z), A) \to H_Z^m(X, A_X) \to H^{n-m}(Z, A_Z)^{\vee} \to 0
\end{equation*}
    
\end{thm}

\begin{proof}
    Let $J^{\bullet}$ be an injective resolution of the constant sheaf $A_Z$ on $Z$, and denote $I^{\bullet}=i_!J^{\bullet}$. Then by Lemmas \ref{2.4.2} and \ref{1.4.3} and the Verdier duality
\begin{multline*}
    H_Z^m(X, A_X) \simeq \Ext^m(I^{\bullet}, A_X) \simeq \Hom_{D^b(X)}(I^{\bullet}, \omega_X[m-n]) \\
    \simeq \Hom_{D^b(X)}(I^{\bullet}[n-m], a^!_X (A))
    \simeq \Hom_{D^b(pt)} (\RR\Gamma_c (X, I^{\bullet})[n-m], A) .
\end{multline*}

     Let $K^{\bullet}$ be an injective resolution of $A$. Then
\begin{equation*}
     \Hom_{D^b(pt)} (\RR\Gamma_c (X, I^{\bullet})[n-m], A) \simeq \Hom_{K(pt)} (\RR\Gamma_c (X, I^{\bullet})[n-m], K^{\bullet}) 
\end{equation*}

    and by Lemma \ref{1.4.1} we get
\begin{equation*}
    \Hom_{K(pt)} (\RR\Gamma_c (X, I^{\bullet})[n-m], K^{\bullet}) \simeq H^{m-n} (\Hom^{\bullet}(\RR\Gamma_c (X, I^{\bullet}), K^{\bullet})).
\end{equation*}

    Now we can apply Theorem \ref{uc} for the complexes $\RR\Gamma_c (X, I^{\bullet})$ and $K^{\bullet}$ and get
\begin{multline*}
    0 \to \Ext (H_c^{n-m+1}(\RR\Gamma_c (X, I^{\bullet})), A) \to H_Z^m(X, A_X) \to \\
    \to \Hom (H_c^{n-m}(\RR\Gamma_c (X, I^{\bullet})), A) \to 0
\end{multline*}

    Finally, notice that
\begin{multline*}
    H_c^{n-m}(\RR\Gamma_c (X, I^{\bullet})) \simeq H^0 (\RR\Gamma_c (X, i_! J^{\bullet}[n-m])) \simeq \\
     H^0(\RR \Hom^{\bullet}(A_X, i_! J^{\bullet}[n-m])) \simeq \Hom_{D^b(X)}(A_X, i_! J^{\bullet}[n-m]) \simeq \\
     \Hom_{D^b(Z)}(i^* A_X, J^{\bullet}[n-m]) \simeq \Hom_{D^b(Z)}(A_Z, J^{\bullet}[n-m]) \simeq H^{n-m}(Z, A_Z) .
\end{multline*}

    Therefore
\begin{equation*}
     0 \to \Ext (H^{n-m+1}(Z, A_Z), A) \to H_Z^m(X, A_X) \to H^{n-m}(Z, A_Z)^{\vee} \to 0 
\end{equation*}
is exact.

\end{proof}

\section{Proof of the main result}

\begin{thm}\label{myt}
    Let $\widetilde{\P}^n = \P(q_0, \ldots, q_n)$ be a weighted projective space with pairwise coprime weights. Let $X$ be a 
 general hypersurface of $\widetilde{\P}^n$ of degree $d$ which is divided by all $q_i$. And let $ i\colon X \rightarrow \widetilde{\P}^n$ be the inclusion map. Then for $k < n-1$ 
    \begin{enumerate}
        \item \begin{equation*}
    H^k(X,\Z) = 
    \begin{cases}
            \Z, & \text{if $k$ is even}\\
            0, & \text{otherwise}
    \end{cases}
\end{equation*}
        \item for even $k=2r$ the pullback map $$i^* \colon H^{2r}(\widetilde{\P}^n, \Z) \rightarrow H^{2r}(X, \Z)$$ is multiplication by $\frac{l_r \cdot l_{n-r}}{l_n}$.
    \end{enumerate}
\end{thm}

\begin{proof}  Consider the long exact sequence for the cohomology with support in $X$:
\begin{center}
$\to H^{k-1}(\widetilde{\P}^n \setminus X, \underline{\Z}) \to H_X^{k}(\widetilde{\P}^n, \underline{\Z}) \to H^{k}(\widetilde{\P}^n, \underline{\Z}) \to H^{k}(\widetilde{\P}^n \setminus X, \underline{\Z}) \to$
\end{center}

$\widetilde{\P}^n \setminus X$ is affine so it has a homotopy type of an $n$-dimensional CW complex \cite{Karcjauskas}. Therefore, for all $k > n$ we have $H^{k}(\widetilde{\P}^n \setminus X, \underline{\Z}) = 0$ and the map from $H_X^{k}(\widetilde{\P}^n)$ to $H^{k}(\widetilde{\P}^n)$ is an isomorphism for $k > n+1$. Therefore 

\begin{equation*}
    H_X^{k}(\widetilde{\P}^n) \cong H^{k}(\widetilde{\P}^n) \text{    for  } k>n+1
\end{equation*}

We required such set of weights and such a degree that generic hypersurface $X$ is smooth and does not pass through the singularities of $\widetilde{\P}^n$. Therefore, we can replace $\widetilde{\P}^n$ with a smooth tubular neighborhood of $X$ to compute $H^k_X(\widetilde{\P}^n, \underline{\Z})^{\vee}$ \cite{iversen}[p.125]. 
It is an orientable manifold of dimension $2n$, so we can apply the Alexander duality Theorem \ref{AD} and obtain the exact sequence  

\begin{equation}\label{es}
     0 \to \Ext (H^{2n-k+1}(X, \underline{\Z}), \Z) \to H_X^k(\widetilde{\P}^n, \underline{\Z}) \to H^{2n-k}(X, \underline{\Z})^{\vee} \to 0 .
\end{equation}

\

Let $k>n+1$. 

If $k$ is odd, then $H_X^k(\widetilde{\P}^n, \underline{\Z}) = 0$, therefore $H^{2n-k}(X, \underline{\Z})^{\vee}=0$ and so $H^{2n-k}(X, \underline{\Z})=0$.

If $k$ is even, then $H_X^k(\widetilde{\P}^n, \underline{\Z}) = \Z$, therefore $H^{2n-k}(X, \underline{\Z})^{\vee}$ can be either $0$, $\Z$ or $\Z/m\Z$. Notice that $\Z/m\Z$ can not be dual of anything, so we are left with two options. So the exact sequence \eqref{es} can either be

\begin{equation*}
    0 \to 0 \to \Z \to \Z \to 0
\end{equation*}
or 
\begin{equation*}
    0 \to \Z \to \Z \to 0 \to 0
\end{equation*}

For every finitely generated abelian group $A$, $\Ext_{\Z}^1(A,\Z)$ is torsion, so \eqref{es} is isomorphic to $0 \to 0 \to \Z \to \Z \to 0$ and $H_X^{k}(\widetilde{\P}^n, \underline{\Z}) \simeq H^{2n-k}(X, \underline{\Z})^{\vee}$.
That proves the first part of the theorem. 

\

Let us prove the second statement. For even $k$ set $2r=k$.  Due to Lemma \ref{pf}, there is a commutative diagram:

\begin{center}

\begin{tikzcd}
H^{2r}(\widetilde{\P}^n, \Z) \arrow{r}{i^*} \arrow[d]              & H^{2r}(X, \Z) \arrow{d}{\simeq} \\
H^{2n - 2r}(\widetilde{\P}^n, \underline{\Z})^{\vee} \arrow{r}{\simeq} & H^{2n - 2r}_X(\widetilde{\P}^n, \underline{\Z})^{\vee}                        
\end{tikzcd}

\end{center}

The map $i^*$ is determined completely by the map from $H^{2r}(\widetilde{\P}^n)$ to $H^{2n-2r}(\widetilde{\P}^n)^{\vee}$. Using \eqref{1} one can see that it is multiplication by $\frac{l_r \cdot l_{n-r}}{l_n}$. 

\end{proof}

\section{Application: The intersection form on second cohomology of Fano manyfold}

Consider a general hypersurface $X$ of weighted degree $d$ in $\widetilde{\P}^{n+1}$ with at worst isolated singularities. It is a $2n$-dimensional manifold, so the cup product defines an $n$-form in the second cohomology group:

\begin{equation*}
    H^2(X, \Z) \otimes \ldots \otimes H^2(X, \Z) \to H^{2n}(X, \Z) \simeq \Z
\end{equation*}

 Note that $H^2(X, \Z) \simeq \Z$ due to the theorem \ref{myt}, so this $n$-form would be simply

\begin{equation}\label{3form}
    (\alpha_1, \ldots , \alpha_n) \to k \cdot \alpha_1 \cdot \ldots \cdot \alpha_n ,
\end{equation}
 
where $k$ is some integer.

\begin{cor}\label{c27}
    For $X$ a general smooth hypersurface of degree $d$ in $\widetilde{\P}^{n+1}$, the $n$-form on the second cohomology induced by the cup product is given by:

    \begin{equation*}
        (\alpha_1, \ldots, \alpha_n) \to \frac{l_{n+1}^{n-1} \cdot d}{l_n^n} \prod\limits_{i=1}^n \alpha_i
    \end{equation*}
\end{cor}

\begin{proof}
    
Recall the map $\phi\colon \P^{n+1} \to \widetilde{\P}^{n+1}$ sending $(x_0: \ldots : x_{n+1})$ to $\phi (x_0: \ldots :x_{n+1}) = (x_0^{q_1}: \ldots :x_{n+1}^{q_{n+1}})$ from the Theorem \ref{t2}.

Denote $X' = \phi^{-1}(X)$ - the preimage of $X$ under map $\phi$. Note that $X'$ is the hypersurface of degree $d$ but in the ordinary projective space. Then the diagram

\begin{center}
    \begin{tikzcd}
X' \arrow[r, hook] \arrow[d] & \P^{n+1} \arrow[d] \\
X \arrow[r, hook]            & \widetilde{\P}^{n+1}  
\end{tikzcd}
\end{center}

leads to the commutative diagram in cohomology 

\begin{center}
\begin{tikzcd}
H^2(\widetilde{\P}^{n+1}) \arrow{r}{i^*} \arrow{rdd}{\phi^*} & H^2(X) \arrow[r] \arrow[d]                                    & H^{2n}(X) \arrow[d] & H^{2n}(\widetilde{\P}^{n+1}) \arrow[l] \arrow{ldd}{\phi^*} \\
                                                      & H^2(X') \arrow[r]                                             & H^{2n}(X')          &                                        \\
                                                      & H^2(\P^{n+1}) \arrow{r} \arrow[u] & H^{2n}(\P^{n+1}) \arrow[u] &                                       
\end{tikzcd}
\end{center}

where maps from $H^2$ to $H^{2n}$ are described $n$-forms. 

We can compute some of the maps in the diagram:

\begin{itemize}
    \item the maps $\phi^*$ are given by Theorem \ref{t2};

    \item the map $H^{2n}(X) \to H^{2n}(X')$ is multiplication by degree of a map $\phi$ and equals the $l_n$ due to Theorem \ref{t3};

    \item the map from $H^2(\P^{n+1})$ to $H^2(X')$ is isomorphism due to the Lefschetz theorem;

    \item the map from $H^{2n}(\P^{n+1})$ to $H^{2n}(X')$ is multiplication by the degree of $X'$

    \item and the map $i^*$ is multiplication by $\frac{l_1 \cdot l_n}{l_{n+1}}$ by the Theorem \ref{myt}.
\end{itemize}

\

Putting it all in the diagram above we have:

\begin{center}
\begin{tikzcd}
H^2(\widetilde{\P}^{n+1}) \arrow{r}{\cdot \frac{l_1 \cdot l_n}{l_{n+1}}} \arrow{rdd}{\cdot l_1} & H^2(X) \arrow[r] \arrow[d]                                    & H^{2n}(X) \arrow{d}{\cdot l_{n+1}} & H^{2n}(\widetilde{\P}^{n+1}) \arrow[l] \arrow{ldd}{\cdot l_n} \\
                                                      & H^2(X') \arrow{r}{\alpha \to d \cdot \alpha^n}                                             & H^{2n}(X')          &                                        \\
                                                      & H^2(\P^{n+1}) \arrow{r}{\alpha \to \alpha^n} \arrow[u, Rightarrow] & H^{2n}(\P^{n+1}) \arrow{u}{\cdot d} &                                       
\end{tikzcd}
\end{center}

This information is enough to compute all remaining maps in the diagram, including the desired $n$-form.

\end{proof}

We can apply this to Fano manifolds of this form. Below is the table for dimensions not greater than 6.

\begin{table}[H]
\caption{ }
\label{table:1}
\begin{center}
\begin{tabular}{| c | c | c | c | c |} 
 \hline
 dim & w.p.s & degree & intersection form multiple & index \\ [0.5ex] 
 \hline\hline
 3 & $\P(1,1,1,1,2)$ & 4 & 2 & 2\\ 
 \hline
 3 & $\P(1,1,1,1,3)$ & 6 & 2 & 1\\ 
 \hline
 3 & $\P(1,1,1,2,3)$ & 6 & 1 & 2\\ 
 \hline\hline
 
 4 & $\P(1,1,1,1,1,2)$ & 4 & 2 & 3\\ 
 \hline
 4 & $\P(1,1,1,1,1,2)$ & 6 & 3 & 1\\ 
 \hline
 3 & $\P(1,1,1,1,1,3)$ & 6 & 2 & 2\\ 
 \hline
 4 & $\P(1,1,1,1,1,4)$ & 8 & 2 & 1\\ 
 \hline
 4 & $\P(1,1,1,1,2,3)$ & 6 & 1 & 3\\  
 \hline
 4 & $\P(1,1,1,1,2,5)$ & 10 & 1 & 1\\  
 \hline\hline
 
 5 & $\P(1,1,1,1,1,1,2)$ & 4 & 2 & 4\\ 
 \hline
 5 & $\P(1,1,1,1,1,1,2)$ & 6 & 3 & 2\\ 
 \hline
 3 & $\P(1,1,1,1,1,1,3)$ & 6 & 2 & 3\\ 
 \hline
 5 & $\P(1,1,1,1,1,1,4)$ & 8 & 4 & 2\\ 
 \hline
 5 & $\P(1,1,1,1,1,2,3)$ & 6 & 1 & 4\\  
 \hline
 5 & $\P(1,1,1,1,1,2,5)$ & 10 & 1 & 2\\  
 \hline\hline
 
 6 & $\P(1,1,1,1,1,1,1,2)$ & 4 & 2 & 5\\  
 \hline
 6 & $\P(1,1,1,1,1,1,1,2)$ & 6 & 3 & 3\\  
 \hline
 6 & $\P(1,1,1,1,1,1,1,2)$ & 8 & 4 & 1\\  
 \hline
 6 & $\P(1,1,1,1,1,1,1,3)$ & 6 & 2 & 4\\  
 \hline
 6 & $\P(1,1,1,1,1,1,1,3)$ & 9 & 3 & 1\\  
 \hline
 6 & $\P(1,1,1,1,1,1,1,4)$ & 8 & 2 & 3\\  
 \hline
 6 & $\P(1,1,1,1,1,1,2,3)$ & 6 & 1 & 5\\  
 \hline
 6 & $\P(1,1,1,1,1,1,2,5)$ & 10 & 1 & 3\\  
 \hline
 6 & $\P(1,1,1,1,1,1,3,4)$ & 12 & 1 & 1\\  
 \hline
 6 & $\P(1,1,1,1,1,1,2,7)$ & 14 & 1 & 1\\  
 \hline
\end{tabular}
\end{center}
\end{table}

From the Table \ref{table:1}, one can see that for different manifolds, the intersection form together with the index also differs for dimensions lower than 3, 4 and 5. However, for two hypersurfaces of dimension 6: degree 12 in $\P(1,1,1,1,1,1,3,4)$ and degree 14 in $\P(1,1,1,1,1,1,2,7)$, both the intersection form and the index are the same. It may happen that those hypersurfaces are isomorphic. Using the Hodge diamond cutter \cite{pieter_belmans_2024_10854199}, one can compute the Hodge diamonds of those manifolds. The Hodge diamonds are different; therefore, those Fanos are not isomorphic, and this counterexample indeed breaks the hypothesis about using the intersection form and index as complete invariants for hypersurfaces in dimension 6.

\medskip

\printbibliography

\end{document}